\documentclass[12pt,letterpaper,titlepage]{amsart}
\usepackage{amsmath, amssymb, amsthm, amsfonts,amscd,xr}

\newcommand{\C}{\mathbb{C}}

\def\beq{\begin{equation}}
\def\eeq{\end{equation}}

\def\arr{\hbox to 20pt{\rightarrowfill}}

\def\a{\alpha}

\def\Ad{\operatorname {Ad}}

\newenvironment{res} 
               {\begin{equation} 
\begin{minipage}{0.85\textwidth}} 
               { \end{minipage}\end{equation} } 
\def\ber{\begin{res} } 
\def\eer{\end{res}} 
 
\numberwithin{equation}{section} 
\newtheorem{thm}{Theorem}[section]

\newtheorem{lemma}[thm]{Lemma} 
\newtheorem{cor}[thm]{Corollary}

\makeatletter 
\def\section{\@startsection {section}{1}{\z@}{3.5ex plus 1ex minus 
    .2ex}{2.3ex plus .2ex}{\large\bf}} 
    \def\subsection{\@startsection{subsection}{2}{\z@}{3.25ex plus 1ex minus 
 .2ex}{1.5ex plus .2ex}{\bf}} 
\makeatother 
\def\pf{{\em Proof}.\, } 
 
\def\bysame{\leavevmode\hbox to3em{\hrulefill}\,}

\def\af{\mathfrak{a}} 

\def\gf{\mathfrak{g}}

\def\kf{\mathfrak{k}}

\def\pf{\mathfrak{p}}

\def\Ze{\mathcal{Z}}

\def\Re{\operatorname{Re}}

\def\Cas{\mathrm{C}}

\makeindex 
\begin{document} 

\title[Holomorphic extension]
{Holomorphic extension of eigenfunctions}
\author{Bernhard Kr\"otz} 
\address{Max-Planck-Institut f\"ur Mathematik\\  
Vivatsgasse 7\\ D-53111 Bonn
\\email: kroetz@mpim-bonn.mpg.de}

\author{Henrik Schlichtkrull}
\address{Department of Mathematics\\ 
University of Copenhagen\\ Universitetsparken 5 \\ 
DK-2100 K\o{}penhavn 
\\ email: schlicht@math.ku.dk} 

\date{\today}
\begin{abstract}{Let $X=G/K$ be a Riemannian symmetric space
of non-compact type. We prove a theorem of holomorphic extension 
for eigenfunctions of the Laplace-Beltrami operator
on $X$, by techniques from the theory of partial 
differential equations.}\end{abstract} 
\thanks{}
\maketitle 

\section{Introduction}
Let $X$ be a Riemannian symmetric space of non-compact type.
Then $X=G/K$, where $G$ is a connected semisimple Lie group 
and $K$ a maximal compact subgroup. We choose the group $G$ 
such that it is contained in a complexification $G_\C$, and 
we denote by $K_\C\subset G_\C$ the complexification of $K$. 
The symmetric space $X_\C=G_\C/K_\C$ carries a natural 
complex structure, and it contains $X$ as a totally real 
submanifold.

We are interested in eigenfunctions of the
Laplace-Beltrami operator $\Delta$ on $X$. Since this operator is
elliptic and $G$-invariant, every eigenfunction 
admits a holomorphic extension to some open $G$-invariant neighborhood
of $X$ in $X_\C$. The $G$-orbits in $X_\C$ are generally difficult
to parametrize, but let us recall that a particular $G$-invariant open
neighborhood $\Xi$ of $X$, for which the orbit structure is 
compellingly simple, has been proposed in \cite{AG}. It is commonly called
the {\it complex crown} of $X$, and it has been thoroughly
investigated in recent years. See for example \cite{B,F,GK,Huck,KS1,KS2,M}.
In the present paper we show that every eigenfunction 
for $\Delta$ extends holomorphically to $\Xi$. 

Our result generalizes a result from \cite{KS2} that every
joint eigenfunction for the full set of invariant differential 
operators on $X$ extends holomorphically to $\Xi$. The proof
given in \cite{KS2} invokes the Helgason conjecture, affirmed in
\cite{KKMOOT} by micro-local analysis. Our proof 
is considerably simpler. The crucial step is an application
of a theorem from the theory of analytic partial differential equations. 
This theorem asserts the existence of 
a holomorphic extension to solutions which
are holomorphic on one side of a non-characteristic surface.

At the end of the paper a further generalization is given
to functions on $G$, which are eigenfunctions for the Casimir
operator and right-$K$-finite.

\section{Notation}
We denote by $\gf=\kf\oplus\pf$ the Lie algebra of $G$ and its
Cartan decomposition. We choose a maximal abelian subspace
$\af$ in $\pf$ and denote by $\Sigma\subset\af^*$ the corresponding
system of restricted roots.
The root spaces in $\gf$ are denoted by $\gf^\alpha$,
where $\alpha\in\Sigma$, and by $\Sigma^+\subset\Sigma$ we 
denote a positive system. The centralizer of 
$\af$ in $K$ is $M=Z_K(\af)$, and the Weyl
group is $W=N_K(\af)/M$, where $N_K(\af)$ is the
normalizer.

Recall the definition of the complex crown $\Xi$ of $X$. 
We set 
$$\Omega:=\{Y\in\af\mid 
|\alpha(Y)|<\frac\pi 2, \forall \alpha\in\Sigma\}\, .$$
Then
$$\Xi:= G\exp(i\Omega) K_\C =\{g\exp(iY)\cdot x_0
\mid g\in G, Y\in\Omega\}\subset X_\C.$$ 
Here $x_0$ denotes the standard base point $eK_\C$ in $X_\C$.

\section{Results}
\begin{lemma}\label{l:crownstructure} The $G$-invariant 
crown $\Xi$ is an open subset of $X_\C$. The surjective map
$$\Phi: G\times\Omega\to\Xi,\quad (g,Y)\mapsto g\exp(iY)\cdot x_0$$ 
is real analytic, and the topology of $\Xi$ is identical to
the quotient topology with respect to this map.

Let $\Omega^+$ be the intersection of $\Omega$ with the positive
open chamber in $\af$, and let $\Xi'=\Phi(G\times\Omega^+)$. 
Then $\Xi'$ is open and dense in $\Xi$, and
$$\Phi': G/M\times\Omega^+\to\Xi',\quad (gM,Y)\mapsto\Phi(g,Y)$$
is a diffeomorphism.
\end{lemma}

\begin{proof} Apart from the statement about the topology of
$\Xi$, this can be found in \cite{KS2}, \S 4.
For the topological statement we need to prove 
that a subset of $\Xi$ is open if its preimage is open.
It suffices to prove the following. Let $z_n\to z\in\Xi$
be a converging sequence. Then there exists a subsequence
of the form $z_j=\Phi(g_j,Y_j)$ with converging sequences
$g_j\to g\in G$ and $Y_j\to Y\in\Omega$. It follows from
\cite{AG}, Propositions 1 and 7, that there exist sequences
$g_n$ in $G$, $k_n$ in $K$  and $Y_n$ in $\Omega$
such that $z_n=\Phi(g_nk_n,Y_n)$, and such that
$g_n$ and $k_n\exp(iY_n)\cdot x_0$ both converge.
By passing to a subsequence, we may assume that
$k_j$ converges, and since $Y\mapsto \exp(iY)\cdot x_0$
is a diffeomorphism of $\Omega$ onto its image, it 
then follows that $Y_j$ converges in $\Omega$.
\end{proof}

\begin{thm}\label{t:holoext} Let  
$f\in C^\infty(X)$ be an eigenfunction for $\Delta$. 
Then $f$ extends to a holomorphic function on $\Xi$.
\end{thm}

\begin{proof} 
As $\Delta$ is elliptic, the regularity theorem 
for elliptic differential operators 
(see \cite{Horm}, Theorem 7.5.1) 
implies that $f$ is an analytic function. As such 
it has an extension to a holomorphic 
function on some open neighborhood $U_0$ of $x_0$ in $\Xi$. 
It follows from the proof in \cite{Horm},
that $U_0$ can be chosen independently of $f$,
that is, {\it every} eigenfunction can be 
holomorphically extended to $U_0$
(the radius of convergence obtained in the proof
depends only on the corresponding radii for the
coefficients of the differential operator).
In particular, it follows from the fact that 
$\Delta$ is $G$-invariant, 
that $L_gf$ extends to $U_0$ for all $g\in G$. 
The union $U$ of the $G$-translated sets $L_{g^{-1}}(U_0)$ 
is then a $G$-invariant open neighborhood of $X$ in $\Xi$, 
to which $f$ extends.

We now consider the open dense subset $\Xi'\subset 
\Xi$ from Lemma \ref{l:crownstructure}. The intersection
$U\cap\Xi'$ is non-empty, open and $G$-invariant. 
Let $Y_0\in\Omega^+$, and for $r>0$
let $B_r$ denote the open ball in $\af$ of radius $r$,
centered at $Y_0$. If $B_r\subset\Omega^+$, then we define an open
set $$T_r=G\exp(iB_r)\cdot x_0\subset \Xi',$$ which we regard
as a $G$-invariant 
`circular tube' in $\Xi'$. We claim that if 
$f$ extends holomorphically to a set containing some 
circular tube $T_r\subset\Xi'$
centered at $Y_0$, then it extends to all circular tubes in $\Xi'$
centered at $Y_0$. Since $Y_0$ was arbitrary, and since
$\Omega^+$ is simply connected, it follows from this 
claim that $f$ extends 
holomorphically from $U\cap \Xi'$ to $\Xi'$.

In order to establish the claim we use Theorem 9.4.7 
of \cite{H}, due to Zerner \cite{Z}. 
We  write $\Delta_\C$ for the extension of $\Delta$
to a $G_\C$-invariant holomorphic differential operator on $X_\C$.
Obviously, the holomorphic extension that we seek will
be an eigenfunction for $\Delta_\C$ on $\Xi$.
It follows from Lemma \ref{l:crownstructure} that
each circular tube $T_r$, for which the closure is contained in
$\Xi'$, has real-analytic boundary $\partial T_r$. 
In order to apply Zerner's theorem it suffices 
to establish that $\partial T_r$
is non-characteristic for $\Delta_\C$, for all such tubes.
By $G$-invariance, it suffices to consider boundary
points $x\in\partial T_r$ with
$x\in\exp(i\Omega)\cdot x_0$.
Recall from \cite{KS2}, p. 207, that when 
$x\in\exp(i\Omega)\cdot x_0$
we have a complex-linear isomorphism
\begin{equation}\label{tangentmap}
 \pf_\C\ni Z\mapsto \tilde Z_x\in T_x\Xi,
\end{equation}
where $\tilde Z$ is the holomorphic vector field on $X_\C$ given by 
$$\tilde Z_x \varphi= L_Z\varphi(x)=\frac{d}{dz} 
\varphi(\exp(-zZ)x)|_{z=0}.$$
In this isomorphism the tangent space at $x$ of the
boundary $\partial T_r$ will then be a real
hyperplane given
by an equation $\Re\zeta(Z)=0$ for some 
cotangent vector $\zeta\in \pf_\C^*$.
Since the tube $T_r$ is $G$-invariant, 
it follows that $\Re\zeta$ annihilates  $Z$
for all $Z\in \pf$, so $\zeta$ is purely imaginary on $\pf$.

Let $(X_j^\alpha)_{\alpha\in\Sigma^+, 1\leq j\leq m_\alpha}$ 
together with $Y_1,\dots,Y_r\in \af$
be an orthonormal basis for $\pf$ such that 
$X_j^\alpha\in \pf_\alpha:=[\gf^\alpha+\gf^{-\alpha}]\cap \pf$.
Here $m_\alpha=\dim \pf_\alpha$ as usual. In the 
universal enveloping algebra we have 
$$\Delta =\sum_{\alpha\in \Sigma^+}\sum_{j=1}^{m_\alpha} (X_j^\alpha)^2
+\sum_{i=1}^r Y_i^2$$
with respect to the right action, where functions on $G/K$ 
are regarded as a
right $K$-invariant function on $G$. Observe that
modulo  $\kf$,
$$\Ad(a)^{-1}(X^\alpha_j)= \cosh(\alpha(\log a)) X^\alpha_j$$
for $a\in A$
(see \cite{KS2} p. 207), and hence
$$R_{X^\alpha_j}f(a)=-[\cosh\alpha(\log a)]^{-1} L_{X^\alpha_j}f(a).$$
It follows that
\begin{equation}\label{e:Delta}
\Delta=\sum_{\alpha\in \Sigma^+}\sum_{j=1}^{m_\alpha} 
[\cosh\alpha(\log a)]^{-2} (X_j^\alpha)^2+\sum_{i=1}^r 
Y_i^2
\end{equation}
at $z=a\cdot x_0\in X$,
with respect to the left action.

By analytic continuation, the same equation 
holds as well for $\Delta_\C$ and $a\in A_\C$.
In particular, at $x=\exp(iY)\cdot x_0$ we obtain
$$
\Delta_\C=\sum_{\alpha\in \Sigma^+}\sum_{j=1}^{m_\alpha} 
[\cos\alpha(Y)]^{-2} [(\tilde X_j^\alpha)_x]^2+\sum_{i=1}^r
[(\tilde Y_i)_x]^2.
$$
Note that the condition that $x$ belongs to the crown
precisely ensures that $\cos\alpha(Y)\neq 0$,
so that the expression makes sense.
As $\zeta$ is purely imaginary on $\pf$,
it follows that all terms in the above sum are $\leq 0$
when applied to $\zeta$.
Thus the principal symbol of $\Delta_\C$
is non-zero at $\zeta$, and
the boundary of $T_r$ is non-characteristic. It follows that
Zerner's theorem can be applied, so that
$f$ extends holomorphically to $\Xi'$.

For the extension to the full set $\Xi$ we shall 
apply Bochner's theorem
(see \cite{severalcomplex}, Theorem 2.5.10). 
>From what we have seen so far, for all $g\in G$ the function 
$$f_g: \af\to \C , \ \ Y\mapsto f(g\exp(Y)\cdot x_0)$$
extends to a holomorphic function on a tubular
neighborhood $\af+i\omega$ of $\af$ in $\af_\C$,
and also to $\af+i\Omega^+$. For elements $w\in N_K(\af)$
we have 
$$
f_{g}(\Ad(w)Y)=f_{gw}(Y).
$$
It follows that $f_g$ extends to a holomorphic
function on each Weyl conjugate of $\af+i\Omega^+$, hence to
$\af+i\Omega'$, where
$\Omega'=\cup \Ad(w)(\Omega^+)$ is the set of
regular elements in $\Omega$. 
Now
Bochner's theorem implies that $f_g$ extends to a 
holomorphic function on the tube over the convex hull of
$\omega\cup\Omega'$, that is, to $\af+i\Omega$. 
Furthermore, $g\mapsto f_g$ is continuous 
into $H(\af+i\Omega)$ (with standard topology), 
since it is continuous into the space
$H(\af+i(\omega\cup \Omega'))$ which by Bochner's theorem
is topologically isomorphic.

Recall that for all
$g,g'\in G$ and $Y,Y'\in\Omega$ we have
$$g\exp(iY)\cdot x_0=g'\exp(iY')\cdot x_0$$ 
if and only if there exists $w\in N_K(\af)$ and $k\in Z_K(Y)$
with $g'=gkw$ and $Y'=\Ad(w^{-1})Y$. It follows easily that
by
$$g\exp(iY)\cdot x_0\mapsto f_g(Y)$$
we obtain a well-defined extension of $f$ on $\Xi$.
The topological statement in the first part of
Lemma \ref{l:crownstructure}
implies that this extension is continuous. Since the extension 
is holomorphic
on $\Xi'$, it must be holomorphic everywhere.
\end{proof}

We list some easy consequences of the preceding theorem and its proof.
>From the Iwasawa decomposition $G=NAK$ associated to
the positive system $\Sigma^+$, we obtain 
the familiar horospherical projection $x\mapsto H(x)\in\af$,
defined by $x\in N\exp H(x)\cdot x_0$ for $x\in X$. 
For each $\lambda\in\a_\C^*$ 
the function
$$x\mapsto e^{\lambda(H(x))}$$
on $X$ is an eigenfunction for $\Delta$, hence extends to a holomorphic
function on $\Xi$. 
We obtain: 

\begin{cor} The projection $H: X\to \af$ extends to a holomorphic map
$\Xi\to \af_\C$. Moreover, $\Xi\subset N_\C A_\C K_\C$.
\end{cor} 

\begin{proof} 
Let $h_\lambda(z)$ denote the analytic continuation
of $e^{\lambda(H(x))}$. Since $h_{-\lambda}(z)=h_\lambda(z)^{-1}$
we conclude that $h_\lambda(z)\neq0$. As $\Xi$ is simply
connected, the analytic continuation of $\lambda(H(x))$
is obtained by taking logarithms, and the first statement
follows. For the last statement, we note that once the
Iwasawa $A$-component allows an analytic continuation,
then so does the $N$-component. Indeed, knowing
the $A$-component, we can determine the 
$N$-component of $x\in X$ from $\theta(x)x^{-1}$, where $\theta$ 
denotes the Cartan involution.
\end{proof}

The preceding corollary  
was obtained for classical groups in \cite{KS1}. The
general case follows from results established in \cite{M}, 
\cite{BHH} and \cite{Huck} with \cite{Huck2}.

Let $\omega\subset\Omega$ be open, convex and $W$-invariant, and let
$T_\omega\subset\Xi$ denote the open set 
$$T_\omega=G\exp(i\omega)\cdot x_0.$$

\begin{cor} Let $f\in C^\infty(X)$ and let 
$P$ be a non-trivial polynomial of one variable.
If $P(\Delta)f$ extends to a 
holomorphic function on $T_\omega$, then so does $f$. 
\end{cor}

\begin{proof} By treating the factors of $P$ successively
we may assume that $P(\Delta)=\Delta-\lambda$. The proof
of Theorem \ref{t:holoext} can then be repeated.
\end{proof}

The following generalization is more far-reaching.
We denote by $\Cas\in \Ze(\gf)$ 
the Casimir element of $\gf$. 

\begin{thm} Let $f\in C^\infty(G)$ be a right $K$-finite 
eigenfunction of $\Cas$. Then $f$ extends to a holomorphic 
function on 
$$\tilde \Xi:= G\exp(i\hat \Omega) K_\C \subset G_\C.$$ 
\end{thm}

\begin{proof} Recall that $f$ being $K$-finite means that 
the translates $R_kf$ by $k\in K$ span a finite dimensional 
space, which is then a representation space for $K$.
We may assume that it is irreducible, and then $f$ is
an eigenfunction for the Casimir element $\Cas_\kf$
of $\kf$, acting from the right. The operator 
$\Cas+2\Cas_\kf$ is elliptic, so it follows that
$f$ is real analytic. The proof of Theorem \ref{t:holoext}
can now be repeated, with the following changes.

In Lemma \ref{l:crownstructure} we replace the map
$\Phi$ by 
$$\tilde \Phi: G\times  \Omega \times K_\C\to \tilde\Xi,\quad
(g,Y,k)\mapsto g\exp(iY)k,$$
and we define $\tilde \Phi'$ as before, but now on
$(G\times\Omega^+\times K_\C)/M$, where $M$ acts on the first
and last factor, from the right and left, respectively. 

The $G$-invariant tubes $T_r\subset\Xi$ 
are replaced by their $G\times K_\C$-invariant preimages
$\tilde T_r=T_rK_\C\subset\tilde\Xi$. The map
$Z\mapsto L_Z$ from $\pf_\C$ onto $T_x\Xi$
in (\ref{tangentmap})  
is replaced by $Z\oplus U\mapsto L_Z+ R_U$ from 
$\gf_\C=\pf_\C\oplus\kf_\C$ onto $T_x\tilde\Xi$. Here
$x\in\exp(i\Omega)$. The cotangent vector $\zeta$,
normal to $\tilde T_r$ at $x$ is zero on $\kf_\C$
and purely imaginary on $\pf$, by the same argument as 
before. 

Since $f$ is a $\Cas_\kf$-eigenfunction,
the action of $\Cas$ on it differs only by a constant 
from that of the operator $\Delta$ described in 
(\ref{e:Delta}). Hence $\partial \tilde T_r$
is non-characteristic for $\Cas$, and the application of 
Zerner's theorem goes through. The rest of the argument is 
essentially unchanged. 
\end{proof}

\end{document}